\documentclass[12pt]{article}
\usepackage[utf8]{inputenc}
\usepackage{tikz}
\usepackage{amsmath}
\usepackage{amssymb}
\usepackage{amsfonts}
\usepackage{amsthm}
\usepackage{hyperref}
\usepackage{comment}
\hypersetup{colorlinks=true,
linkcolor=cyan,
filecolor=magenta,
urlcolor=blue,}
\usepackage[margin=1in]{geometry}

\newcommand{\AR}[1]{{\bf [~AR:\ } {\em \textcolor{blue}{#1}\bf~]}}
	
\newcommand{\EM}[1]{{\bf [~EM:\ } {\em \textcolor{magenta}{#1}\bf~]}}

\newcommand{\ecc}{\textrm{ecc}}
\newcommand{\dist}{\textrm{dist}}
\newcommand{\diam}{\textrm{diam}}

\newtheorem{theorem}{Theorem}[section]
\newtheorem{corollary}[theorem]{Corollary}
\newtheorem{definition}[theorem]{Definition}
\newtheorem{lemma}[theorem]{Lemma}

\newtheorem{proposition}[theorem]{Proposition}
\newtheorem{remark}[theorem]{Remark}

\title{The Explorer-Director Game on Graphs}
\author{Pat Devlin\footnote{Department of Mathematics, Yale University, patrick.devlin@yale.edu}, Erin Meger\footnote{Laboratoire d'alg\`ebre, du combinatoire et d'informatique math\'ematique, Universit\'e du Qu\'ebec \`a Montr\'eal, mckenna\_meger.erin\_kathleen@courrier.uqam.ca}, Abigail Raz\footnote{Department of Mathematics, University of Nebraska, Lincoln, araz2@unl.edu}, Polymath REU Participants\footnote{Ben Bates, Tyler Beauregard, Daniel Brey, Elaine Danielson, Mayur Lahane, Angela Li, Mohamed Lotfi, Stephanie Marsh, Oskar Mathiasen, Audrey McMillion, Dhruva Mehta, Joseph Mescall, Jonah Nan, Jarod Palubiski, Cac Phan, Isabel Sacksteder, Danielle Solomon, João Tavares, Lincoln Updike, Shannon Vogel, James Wash, Charles Weng, Erica Zhang, Mengxue Zhang}}

\begin{document}
\maketitle

\begin{abstract}
    The Explorer-Director game, first introduced by Nedev and Muthukrishnan, can be described as a game where two players---Explorer and Director---determine the movement of a token on the vertices of a graph. At each time step, the Explorer specifies a distance that the token must move hoping to maximize the amount of vertices ultimately visited, and the Director adversarially chooses where to move token in an effort to minimize this number.  Given a graph and a starting vertex, the number of vertices that are visited under optimal play is denoted by $f_d(G,v)$.

    In this paper, we first reduce the study of $f_d (G,v)$ to the determination of the minimal sets of vertices that are \textit{closed} in a certain combinatorial sense, thus providing a structural understanding of each player's optimal strategies.  As an application, we address the problem on lattices and trees. In the case of trees, we also provide a complete solution even in the more restrictive setting where the strategy used by the Explorer is not allowed to depend on their opponent's responses.  In addition to this paper, a supplementary companion note will be posted to arXiv providing additional results about the game in a variety of specific graph families.

    
\end{abstract}


\section{Introduction}

Consider a game with two players---an Explorer and a Director---played on a graph $G$. The game is played by moving a token between the vertices of $G$ in the following manner: on any given turn of the game, the Explorer selects a distance to move the token, and the Director goes on to move the token to a vertex that is that distance away from the token's present position. The goal of the Explorer is to maximize the number of vertices the token visits, while the Director's goal is to minimize that same number.The game generally ends when the Explorer cannot effectively call any sequence of distances that would force a visit to an unvisited vertex.

First introduced by Nedev and Muthukrishnan as the Magnus-Derek Game, this game is intended to simulate the behavior of a \emph{Mobile Agent (MA)} exploring a ring network where there is inconsistent global sense of direction \cite{NM}. Through focusing on the specific case in which the game is played on a circular table with $n$ positions labelled consecutively, Nedev and Muthukrishnan had established that the number of vertices of this cycle visited assuming optimal play, which they denoted $f^*(n)$, would be: \[f^*(n)=\begin{cases}
n & \text{if } n=2^k \text{ for some } k>0 \\
\frac{n(p-1)}{p} & \text{if } n=pm, \text{ where } p \text{ is the smallest odd prime factor of } n \text{ and } p>2
\end{cases} \]

 Follow-up research on the game, such as the ones done by Hurkens et al., Nedev, and Chen et al., has advanced alternative algorithms that achieve $f^*(n)$ more quickly with reduced computational complexity \cite{HPW, N10, N14, CLST}. Chen et al. have also developed strategies for other game variants, while others like Gerbner brought about a new angle by introducing an algebraic generalization for the original game, in which positions $g$ $\in$ $G$ are seen as elements of some cyclic group, $G$ \cite{CLST, G}.  Asgeirsson and Devlin further built on this idea \cite{AD}.  Through relating the game to combinatorial group theory and the notion of \emph{twisted subgroup}, they expanded the scope of the game beyond cyclic groups to include other general finite groups, and named the game as the Explorer-Director Game -- the generalized version of the original game which we will be focusing on for our research \cite{AD}.

In this paper, we make progress by developing strategies in various types of graphs, introducing new variants of the game, and constructing alternative algorithms which determine game completion and the minimal solution.

In a graph $G$, the distance between a pair of vertices $u,v \in V(G)$ to be the length of the shortest path between them. For the explorer director game on a graph $G$ with starting vertex $v$, we define the explorer-director number, $f_d(G,v)$, as the  number of visited vertices at the end of the game. 


We begin in Section~\ref{closed}, where we introduce the notion of a \emph{closed set} which in turn implies generic bounds on $f_d(G,v)$. We follow this with Section~\ref{graphs} where we provide specific bounds for lattices and trees. Our final results end with Section~\ref{nonadaptive} where we consider a nonadaptive strategy.  In a subsequent arXiv companion, we include additional results for specific graphs including a full solution for triangular lattices. For ease of reading, these results have been separated from the main results presented here. The results in the arXiv companion allow insight into the game by highlighting important strategies used. We have included the most useful and novel techniques and game strategies in this paper.

An interesting result that shows why we care to look at classes of graphs. These values are not bounded based on subgraphs or vertex numbers; so different classes of graphs could give bounds on larger classes.

\begin{corollary}
For any $b \geq 2$ and $2 \leq k \leq b$, there exists a connected graph $G$ on $b$ vertices such that $f_d(G)=k$.
\end{corollary}

For the simple existence of this, consider a lollipop graph with an $b-k+2$ complete graph and an $k-2$-path graph. There are a number of constructions that give this result.

\section{Closed Subsets}\label{closed}

In this section, we investigate the structure of the set of visited vertices at the end of an optimally played game.

\begin{definition} 
On a graph $G=(V,E)$, we define a non-empty subset $U \subseteq V$ to be \emph{closed} when every vertex $u \in U$ has the property that for every vertex $v \in V$, there exists some vertex $x \in U$ such that $\dist(u,v) = \dist(u,x)$.
\end{definition}

Our notion of closed sets is very similar to another notion of ``distance closed" sets (see \cite{janakiraman2012structural}), where the distances are calculated only using the induced subgraph on $U$. We allow all distances to be calculated using the ambient graph.

In the Explorer Director game, whenever the token is on a vertex of a closed set, regardless of the distance called by the Explorer, the Director can always find a vertex in the closed set to visit. Notice that closed sets exist in every graph, by using the trivial closed set $U=V$, and by definition the empty set is not closed.

\begin{theorem}
For any graph $G$ we have $\min_{v \in V}f_d(G,v)$ is the cardinality of the minimal closed set in $G$.
\end{theorem}

As the Director can always keep the token within a closed set we know that $f_d(G,v)$ is at most the cardinality of the minimal closed set containing $v$. The lower bound is implied by the following more general result.




\begin{theorem}\label{closedthm}
At the end of the Explorer-Director game, the set of all visited vertices, $U$, always contains a closed subset.
\end{theorem}

\begin{proof} 
Let $U$ be the set of visited vertices at the end of the Explorer-Director game on a graph $G(V,E)$. Let $X$ be the set of all vertices $u \in U$ such that there is some vertex $v\in V$ such that for every vertex $w \in U$ we have $\dist(u,v)\neq \dist(u,w)$. Thus $X= \emptyset$ if and only if $U$ is closed. Setting $U = U_1$ and $X= X_1$ let $U_{i+1}=U_i \setminus X_i$ for all $i \ge 1$ (where $X_i$ is the set of all vertices $u \in U_i$ such that there is some vertex $v\in V$ such that for every vertex $w \in U_i$ we have $\dist(u,v)\neq \dist(u,w)$).

We claim that if $v \in \bigcup_{i \le n} X_i$ and the token is on vertex $v$ then the Explorer can force the token out of $U$ in at most $n$ steps. Clearly, this claim holds for $n=1$ by the definition of $X$. We will assume, for the sake of induction, that the statement holds for all vertices in $\bigcup_{i \le k} X_i$. Assume the token is on some vertex $v\in X_{k+1}$. Note that by definition of $X_{k+1}$ the Explorer can force the token out of $U_{k+1}$ in one move. Additionally, as $U \setminus U_{k+1}= \bigcup_{i \le k}X_i$ our inductive hypothesis implies that the Explorer can force the token out of $U$ in at most $n+1$ steps, as desired.

Now we claim that $U_n$ is closed for some $n$, which implies our result as $U_n \subseteq U$ for all $n$. If this is not the case then $U_n = \emptyset$ for some $n$. However, this then implies that $U = \bigcup_{i \le n}X_i$, so by our previous paragraph if the token is in any vertex in $U$ then the Explorer can force the token out of $U$ in at most $n+1$ steps. This contradicts our assumption that $U$ is the set of visited vertices at the end of a game. Thus there must be some $U_n \subseteq U$ which is closed, as desired.

\end{proof}

Theorem \ref{closedthm} has the following nice corollary which provides a lower bound on $f_d(G)$.

\begin{corollary}\label{ecc}
If in a game there exists a set of vertices $A \subseteq V$ which the Explorer can force the Director to visit at least once, then $f_d(G) \geq \min_{v \in A}\ecc(v) + 1$.
\end{corollary}

This simply follows from Theorem \ref{closedthm} and the fact that for any vertex $v$ any minimal closed set containing $v$ must have size at least $\ecc(v)+1$. 



\section{Specific Graphs}\label{graphs}

The problem of the Explorer-Director game was originally defined on an additive group modulo $n$, which is isomorphic to playing the game on a cycle with $n$ vertices \cite{NM}. Here we further explore additional graph families. Additional results on other graph families can be found in the arXiv companion. 

\subsection{Lattices} 





Let $L_{n,m}$ denote an $n$ by $m$ rectangular lattice with $n$ vertices in each row and $m$ vertices in each column, so that $|V(L_{n,m})|=nm$. Since $L_{n,m}$ is isomorphic to $L_{m,n}$ we will always assume $n \ge m$. We will specify a vertex $v$ by $(x, y)$ to indicate its position in the lattice when arranged in the usual way in the plane. That is, the bottom-left corner will be $(1, 1)$, the top-left corner will be $(n, 1)$, and the top-right corner will be $(n, m)$. 

\begin{theorem}
For the graph $L_{n,m}$ with $n,m\ge 2$ if one of $n$ or $m$ is odd, then for any starting vertex $v$,

\[f_d(L_{n,m}, v) = n+m - 1\]

Otherwise, we have  \[n+m \leq f_d(L_{n,m}, v) \leq 2m+n-2\]
\end{theorem}

In a few small cases when $n$ and $m$ are both even we know the lower bound holds with equality. For instance, $L_{2,2} \cong C_4$ and so we know $f_d(C_4, v) = 4$. It is also simple to show that $f_d (L_{4,4},v) = 8$ if $v$ is along one of the outer sides of the lattice. The Explorer can force the token to at least 8 distinct vertices and that the top and bottom rows form a closed set as do the left and right-most columns. However, we do not know the exact answer in general. 


\begin{proof}

Regardless of the size of the lattice and starting vertex Corollary \ref{ecc} provides the lower bound since given any starting vertex the explorer can call the eccentricity of that vertex to force the token to a corner and $n+m-2$ is the eccentricity of any corner vertex. When playing on $L_{n,m}$, all vertices that will be visited can be predetermined by the Director regardless of the Explorer's choices. Director will have two strategies based on the parity of $n$ and $m$. 

We will first focus on the case when one of $n$ and $m$ are odd.  Note that in this case there are at most 2 vertices of minimum eccentricity; we will call these central vertices. Given any starting vertex $v$ let $a$ be a corner vertex furthest from $v$ and let $b$ be the corner vertex antipodal to $a$. Let $A$ be a path of length $n+m-1$ which passes through $v$ and the central vertices and has endpoints $a$ and $b$. See Figure~\ref{FIGevenstrat} for two examples of such a path (the central vertex or vertices are labeled with $c$).


\begin{figure}[h!]
\begin{center}
        \begin{tikzpicture}
            \foreach \x in {0,...,4} {
                \draw [black, thick] (\x, 0)--(\x, 4);
                \draw [black, thick] (0, \x)--(4, \x);
            }
            
            \foreach \x in {0,...,4} {
                \foreach \y in {0,...,4} {
                    \draw [black, fill=white] (\x, \y) circle [radius=0.25] node () {};
                }
            }
            
                \draw [red, fill=lightgray!50] (1, 0) circle [radius=0.25] node () {$v$};
                \draw [red, fill=lightgray!50] (2, 0) circle [radius=0.25] node () {};
                \draw [red, fill=lightgray!50] (2, 1) circle [radius=0.25] node () {};
                \draw [red, fill=lightgray!50] (2, 2) circle [radius=0.25] node () {$c$};
                \draw [red, fill=lightgray!50] (2, 3) circle [radius=0.25] node () {};
                \draw [red, fill=lightgray!50] (2, 4) circle [radius=0.25] node () {};
                \draw [red, fill=lightgray!50] (3, 4) circle [radius=0.25] node () {};
                \draw [red, fill=lightgray!50] (4, 4) circle [radius=0.25] node () {$a$};
                \draw [red, fill=lightgray!50, thick] (0, 0) circle [radius=0.25] node () {$b$};
            
            \foreach \x in {0,...,4} {
                \draw [black, thick] (8, \x)--(11, \x);
            }
            \foreach \x in {0,...,3} {
                \draw [black, thick] (8+\x, 0)--(8+\x, 4);
            }

            \foreach \x in {8,...,11} {
                \foreach \y in {0,...,4} {
                    \draw [black, fill=white] (\x, \y) circle [radius=0.25] node () {};
                }
            }
            
                \draw [red, fill=lightgray!50] (8, 4) circle [radius=0.25] node () {$a$};
                \draw [red, fill=lightgray!50] (9, 4) circle [radius=0.25] node () {};
                \draw [red, fill=lightgray!50] (9, 3) circle [radius=0.25] node () {};
                \draw [red, fill=lightgray!50] (9, 2) circle [radius=0.25] node () {$c$};
                \draw [red, fill=lightgray!50] (10, 2) circle [radius=0.25] node () {$c$};
                \draw [red, fill=lightgray!50] (10, 1) circle [radius=0.25] node () {$v$};
                \draw [red, fill=lightgray!50] (11, 1) circle [radius=0.25] node () {};
                \draw [red, fill=lightgray!50, thick] (11, 0) circle [radius=0.25] node () {$b$};

\end{tikzpicture}
\caption{An example of a path $A$ on the lattices $L_{5,5}$ and $L_{5,4}$}\label{FIGevenstrat}
\end{center}
\end{figure}
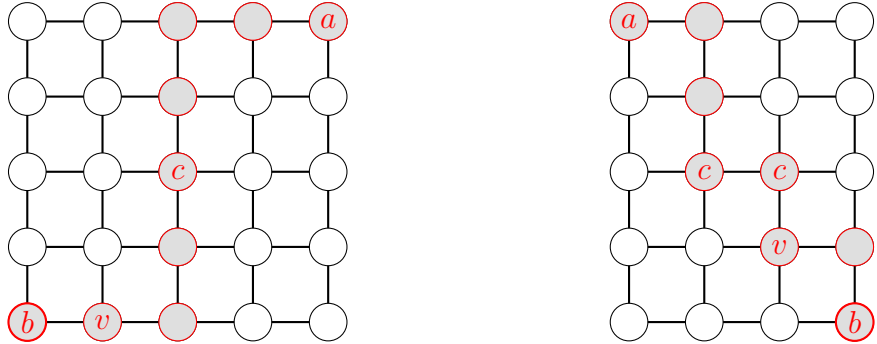

It is not too hard to check that such a path will always be a closed set proving that $f_d(G,v)\le n+m-1$.

We now consider the case when both $n$ and $m$ are even. The key difference is that now $L_{n,m}$ has four central vertices making up a central square. This change allows the Explorer to force the token to visit all four corners, while in the previous case the Director could always ensure only two corners are visited. 

Here we can slightly improve the $n+m-1$ lower bound implied by Corollary \ref{ecc}. Given a particular vertex $v$ let $N_d(v)$ be all the vertices of distance $d$ from $v$. Let $u$ and $w$ be two non-antipodal corners in $L_{n,m}$. Note that the sets $N_0(u), N_2(u),\ldots, N_{n+m}(u), N_0(w), N_2(w), \ldots,$ $N_{n+m}(w)$ are all disjoint (we are only looking at vertices of even distance from either $u$ or $w$). Since the Explorer can ensure all corner vertices are visited they can also ensure at least one vertex from each of these $n+m$ disjoint sets are visited implying $f_d(L_{n,m},v) \ge n+m$ for any starting vertex $v$.

In this case our upper bound comes from noting that for any given starting vertex $v = (x,y)$ the set 
\[A = \{(i,1): 1 \le i \le m\} \cup \{(i,n): 1 \le i \le m\} \cup \{(x,i): 1 \le i \le m\}\]
is closed. See Figure~\ref{FIGlatub} for an example of such a set. Note that this set clearly has size $2m+n-2$ and contains $v$ giving $f_d(L_{n,m},v) \le 2m+n-2$.

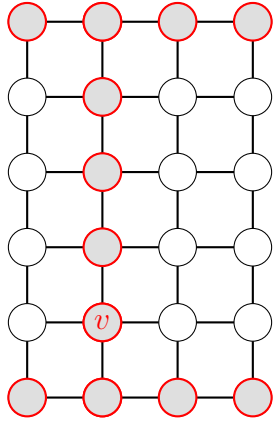
\begin{figure}[h!]
 \begin{center}
        \begin{tikzpicture}
            \foreach \x in {0,...,5} {
                \draw [black, thick] (0, \x)--(3, \x);
            }
             \foreach \x in {0,...,3} {
                \draw [black, thick] (\x, 0)--(\x, 5);

            }
            \foreach \x in {0,...,3} {
                \foreach \y in {0,...,5} {
                    \draw [black, fill=white] (\x, \y) circle [radius=0.25] node () {};
                }
            }

            
            \foreach \x in {0,...,3} {
                    \draw [red, fill=lightgray!50, thick] (\x, 5) circle [radius=0.25] node () {};
                    \draw [red, fill=lightgray!50, thick] (\x, 0) circle [radius=0.25] node () {};
            }
            \foreach \y in {0,...,5} {
                    \draw [red, fill=lightgray!50, thick] (1, \y) circle [radius=0.25] node () {};
            }
            \draw [red, fill=lightgray!50, thick] (1, 1) circle [radius=0.25] node () {$v$};

        \end{tikzpicture}
    \end{center}
  \caption{An example of the set $A$ on the lattice $L_{6,4}$}\label{FIGlatub}
    \end{figure}
We can do slightly better when $n=m$ and the starting vertex is in a particular position. In this case, for instance, the set 
$$B = \{(x,1): 0 \le x \le n\} \cup \{(x,n): 0 \le x \le n\} \cup \{(n-1, x): 3\le x \le n-2\}$$
is closed and of size $3n-4$. Thus if $v \in B$ we have $f_d(L_{n,n},v) \le 3n-4$. We show an example of the set $B$ for the lattice $L_{6,6}$ in Figure~\ref{FIGoddstrat}.
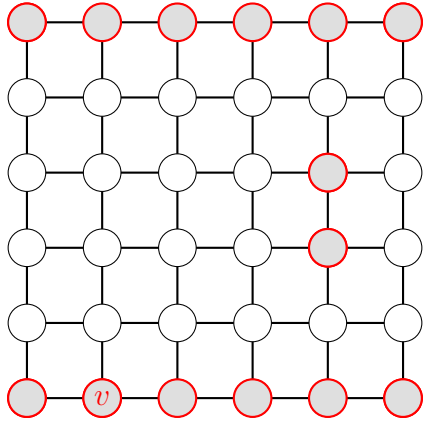
\begin{figure}[h!]
 \begin{center}
        \begin{tikzpicture}
            \foreach \x in {0,...,5} {
                \draw [black, thick] (\x, 0)--(\x, 5);
                \draw [black, thick] (0, \x)--(5, \x);
            }
            \foreach \x in {0,...,5} {
                \foreach \y in {0,...,5} {
                    \draw [black, fill=white] (\x, \y) circle [radius=0.25] node () {};
                }
            }

                \draw [red, fill=lightgray!50, thick] (0, 0) circle [radius=0.25] node () {};
                \draw [red, fill=lightgray!50, thick] (0, 5) circle [radius=0.25] node () {};
                \draw [red, fill=lightgray!50, thick] (5, 0) circle [radius=0.25] node () {};
                \draw [red, fill=lightgray!50, thick] (5, 5) circle [radius=0.25] node () {};

            
            \foreach \y in {0,...,5} {
                    \draw [red, fill=lightgray!50, thick] (\y, 0) circle [radius=0.25] node () {};
                    \draw [red, fill=lightgray!50, thick] (\y, 5) circle [radius=0.25] node () {};
            }
            \foreach \x in {2,3} {
                    \draw [red, fill=lightgray!50, thick] (4,\x) circle [radius=0.25] node () {};
                    \draw [red, fill=lightgray!50, thick] (4, \x) circle [radius=0.25] node () {};
            }
            \draw [red, fill=lightgray!50, thick] (1, 0) circle [radius=0.25] node () {$v$};

        \end{tikzpicture}
    \end{center}
  \caption{An example of the set $B$ on the lattice $L_{6,6}$}\label{FIGoddstrat}
    \end{figure}

\end{proof}

\subsection{Trees}

In this section, we consider the Explorer-Director game played on a tree. We will give an equality bound for $f_d(G,v)$, for any starting vertex $v \in V(G)$ in any tree $G$.

First, we consider the following lemma to provide a lower bound and introduce the strategy we will use in the proof of the main result. In optimal play, the Director will always return to a visited vertex whenever possible. In the following strategy, we show how the Explorer can take advantage of this optimal strategy for the Director.

\begin{lemma}\label{l_tree}
Let $G$ be a tree. Then for all vertices $v$ of $G$, $f_d(G, v) \geq diam(G) + 1.$
\end{lemma}

\begin{proof}

Let the diameter of the tree be $d$. Consider a set of vertices $A$ such that for each vertex in $a \in A$ there exists some vertex $v$ such that $\dist(a,v)=$ $i$ for each $i$ in $\{1,...,d\}$. The set $A$ must contain at least two vertices: the two endpoints of a path along the diameter.  Every vertex in $A$ must be a leaf. 

Whenever the Director occupies a vertex in $A$, there exist at least $d$ vertices with distinct distances, thus the Explorer can force the Director to visit at least one vertex of each distance. Whenever the Director occupies a vertex outside of $A$, the Explorer can force the Director back into $A$, and the game-play continues in this manner until $d+1$ vertices have been visited. 
There are $d+1$ vertices along the diameter path.

\end{proof}


Throughout this section, we will use the strategy of moving from a set of vertices to a new vertex, and then returning to the  initial set. In this way, the Explorer can force the Director into new vertices. We say that $\mathcal{P}$ is the set of all paths $P$ in $G$ that have length $\diam(G)$, and use this notation in the rest of the section.

\begin{remark}\label{r_Director_diam}
Above, we have shown that the Explorer can force the director to visit all the vertices on a specified diameter path. In an optimal strategy, the Director will always remain on a particular path $P \in \mathcal{P}$. Otherwise, the set of visited vertices will included the final diameter path that the Explorer has forced, as well as the initial vertices the Director had visited.
\end{remark}



Considering Remark~\ref{r_Director_diam}, we now consider the main result of the section. When playing the original geodesic distance version of the Explorer-Director game, we have found the precisely the number of vertices that must be visited in optimal gameplay. 

\begin{theorem}\label{tree}
Given a tree $G$ with set of diameter paths $\mathcal{P}$. For a starting vertex $v \in V(T)$, let $u \in V(P)$ for some $P \in \mathcal{P}$ , and we let $\ell =\min_{u \in P \in \mathcal{P}} \dist(u,v)$. Then,
\[f_d(G,v) = \diam(G)+\ell+1\]
\end{theorem}

\begin{proof}

We begin with the proof of the upper bound. Let $G$ be a tree with set of diameter paths $\mathcal{P}$, and specified started vertex $v$. 

First we note, that whenever the Director occupies a vertex, $v$ on some diameter $P$ of the tree $G$, they will always remain on vertices in $P$, since there exists a vertex on $P$ of every distance up to $\ecc(v)$, thus remaining on $P$ will minimize the number of vertices visited. Whenever the Director occupies a vertex that is not on $P$ for some diameter path $P \in \mathcal{P}$, they will choose to move towards some such $P$. Once they have entered some path $P$, they will never leave, and thus the most number of vertices that can be visited are the number of vertices in $P$, which is $\diam(G)+1$, along with the initial vertices required to reach $P$, which is at most $\ell$ as defined in the statement of the theorem, thus giving the upper bound.

Consider now the optimal strategy of the Explorer. Let $u \in V(P), P \in \mathcal{P}$ be the particular vertex closest to $v$ that is on a diameter path. For each vertex along the geodesic between $u,v$, we claim the explorer can force a visit to each vertex. Beginning on $v$, the Explorer uses a similar strategy as used in Lemma~\ref{l_tree}. For each $i \in \{1,..., \dist(u,v)\}$, the Director will move towards $u$ to minimize the number of vertices visited. The Explorer however cannot continue to call $1$ each turn, since the Director would then move back to the previously visited vertex $v$. The Director will never move to a vertex off this path in order to minimize the number of vertices visited. If the Director chooses a vertex not on the geodesic between $u$ and $v$, the Explorer can continue this strategy, and force the Director to visit even more vertices. Thus, the Explorer must return to $v$ after each new vertex along the path of $\dist(u,v)$, and can force the Director to visit one new vertex every other round. This completes the lower bound.

\end{proof}

\section{Nonadaptative strategies}\label{nonadaptive}

In this section, we consider a case in which the Explorer has a strategy which is independent of the Director's choices. We require the following preliminary definition.



\begin{definition}
We define a nonadaptative strategy $S$ as a finite sequence of natural numbers, and its score denoted $f_d^S(G,v)$ as the minimum number of visited vertices if the explorer plays according to $S$ starting at $v$. We say a nonadaptive strategy, $S$, is optimal if $f_d^S(G,v) = f_d(G,v)$.
\end{definition}

There does not necessarily exist an optimal nonadaptative strategy for trees when the starting vertex is not in a longest path. For example, consider the following tree, where $C$ is the center of a sufficiently long path.
\begin{center}
    \begin{tikzpicture}
    \draw [black, thick] (0,0) -- (4, 0);
    \draw [black, fill=white] (0, 0) circle [radius=0.25] node (a) {a};
    \draw [black, fill=white] (1, 0) circle [radius=0.25] node (b) {};
    \draw [black, fill=white] (2, 0) circle [radius=0.25] node (c) {};
    \draw [black, fill=white] (3, 0) circle [radius=0.25] node (d) {};
    
    \draw [black, thick] (4,0) -- (5,1);
    \draw [black, thick] (4,0) -- (5,-1);
    \draw [black, fill=white] (4, 0) circle [radius=0.25] node (L) {C};
    
    \draw [black, thick] (5,1) -- (9, 1);
    \draw [black, fill=white] (5, 1) circle [radius=0.25] node (L) {};
    \draw [black, fill=white] (6, 1) circle [radius=0.25] node (L) {};
    \draw [black, fill=white] (7, 1) circle [radius=0.25] node (L) {};
    \draw [black, fill=white] (8, 1) circle [radius=0.25] node (L) {};
    \draw [black, fill=white] (9, 1) circle [radius=0.25] node (L) {};
    
    \draw [black, thick] (5,-1) -- (9,-1);
    \draw [black, fill=white] (5,-1) circle [radius=0.25] node (L) {};
    \draw [black, fill=white] (6,-1) circle [radius=0.25] node (L) {};
    \draw [black, fill=white] (7,-1) circle [radius=0.25] node (L) {};
    \draw [black, fill=white] (8,-1) circle [radius=0.25] node (L) {};
    \draw [black, fill=white] (9,-1) circle [radius=0.25] node (L) {};
    \end{tikzpicture}
\end{center}
Note that by Theorem \ref{tree} we know $f_d(G,a) = |V(G)|$ for such a graph. However, in any nonadaptive strategy if the explorer ever plays a distance other than 1, then the director can ensure that some vertex in $G$ is never visited. Clearly, if the director only ever plays distance 1 then the explorer can restrict the token to only $a$ and its neighbor. Thus there is no optimal nonadaptive strategy for this graph.

 We will show below that for a particular class of trees there is an optimal nonadaptive strategy when we start at a vertex in a longest path.

Recall that $N_k(v)$ denotes all vertices of distance $k$ from $v$ and any vertex of minimal eccentricity is a central or center vertex. We will use the following standard proposition on trees see, for example, \cite{Knuth}.

\begin{proposition}
A tree has at most two centers.
\end{proposition}

For the following results, we focus on trees with only one center, however, the case for bi-centered trees is analagous.

\begin{theorem}
Let $T$ be a tree with center $c$. If every component in $T\setminus \{c\}$ contains at most one endpoint of a longest path of $T$ and the starting vertex, $v$ is contained in some longest path then there is an optimal nonadaptive strategy. 
\end{theorem}

\begin{proof}

Let us assume, without loss of generality, that the starting vertex is an endpoint, otherwise the explorer starts by calling the eccentricity of the starting vertex. Let $L > 3$ be the length of the longest path and $R$ the radius, that is the eccentricity of the center. The $L \le 3$ case is simple to check. Note that by Theorem \ref{tree} $f_d(G,v) = L+1 = 2R +1$ for $v$ being an end point.

We define the subsequence $s_i = (i, L-i, i, L-i)$ and we define our strategy as the concatenated subsequences $S = (s_1, s_2,\ldots, s_{R-1}, R)$. We claim $f_d^S(G,v)= L+1$.

We will show that after $s_1$ is called 4 vertices will have been explored, after each subsequent $s_i$ is called two new vertices are explored, and the token will visit the center on the last move. 

If the token is at an endpoint of a longest path then for any $1\le i \le R-1$ we know that after the Explorer calls $i$ followed by $L-i$ the token will again be at a different endpoint of a longest path. Thus after $s_1$ is called we know the token must have visited 4 distinct vertices. Furthermore, we know that after calling $s_{i}$ the token will have visited at least two vertices that were not visited in the first $s_{i-1}$ subsequences. This is because after $i$ is called the first time in the $s_i$ subsequence the token must move to a vertex of distance $i$ from an endpoint and after $L-i$ is called the token must return to an endpoint of a longest path. Therefore, before the last move we have explored $4 + 2(R-2) = 2R$ vertices and the token is at an endpoint. Finally, the explorer calls $R$ the token is forced to the center, and the number of visited vertices is now $2R+1 = f_d(G,v)$ as desired.
\end{proof}

The above shows that for a variety of trees, the Explorer can use a completely nonadaptive strategy to visit the optimal number of vertices.  However, the strategy provided requires the Explorer to revisit the same leaves many times, so essentially only half the steps take us to new vertices.  In the case of paths, however, we show that these revisits are unnecessary.

\begin{theorem}\label{nonadaptive-paths}
If $G$ is a path, then the Explorer has a nonadaptive strategy that visits each vertex exactly once, regardless of the Director's choices.
\end{theorem}
\begin{proof}
We provide a construction showing that for any starting vertex, the Explorer can efficiently explore the path by repeatedly declaring numbers that are so large that the Director never has any choice for where the token should go.  We show two cases for the parity of the length of the path. If the path has an odd number of vertices, then the Director would necessarily have a choice whenever the token is on the center vertex. 

We first address the case that $G$ has an even number of vertices, labeled $v_1, v_2, \ldots, v_{2k}$. The strategy is to declare distance $1$ when the token is at $v_1$, and afterwards to alternate between declaring distances $k$ and $k-1$, as shown in Figure \ref{figure:even_paths}.
\begin{figure}[h!]
    \centering
    \includegraphics[width=.9\textwidth]{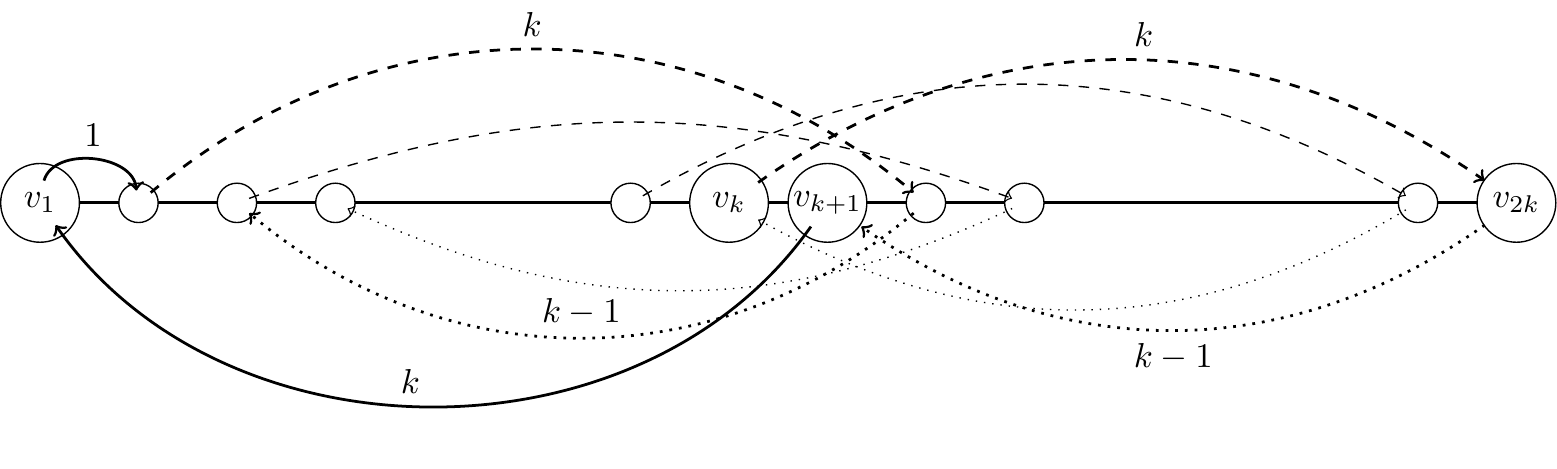}
    \caption{Nonadaptive strategy in even paths, alternating steps of size $k$ (dashed, drawn above the path) and $k-1$ (dotted, drawn below)}
    \label{figure:even_paths}
\end{figure}

Each time the distances $k$ or $k-1$ are declared, the Director always has exactly one choice for where to move the token, since the distance $k$ is declared at each of the two center vertices.  
Moreover, as can be seen, the parity of the path conveniently makes it so that if the token starts at $v_1$, then by the time it arrives at $v_k$, the only unvisited vertices will be $v_{2k}$ and $v_{k+1}$.  Thus, declaring distance $k$ followed by $k-1$ forces the token to each of these in turn.

When the path has an odd number of vertices, the unique center vertex prohibits us from being able to simply force the token as before, and our strategy is necessarily more involved.  For this, we number the vertices $v_0, v_1, \ldots, v_{2k}$, with $v_k$ being the center vertex.  We begin at a vertex $v_x$ with $0 \leq x < k$ and we exhibit a forcing strategy that visits each vertex exactly once and ending on $v_{k}$.  If---on the other hand---the vertex instead starts at $v_k$, notice that the Explorer could start by saying any distance they like and then simply follow our strategy, omitting the final step of moving to $v_k$.  Our strategy is described in two phases, as shown in Figure \ref{figure:odd_paths}.
\begin{figure}[h]
    \centering
    \includegraphics[width=.9\textwidth]{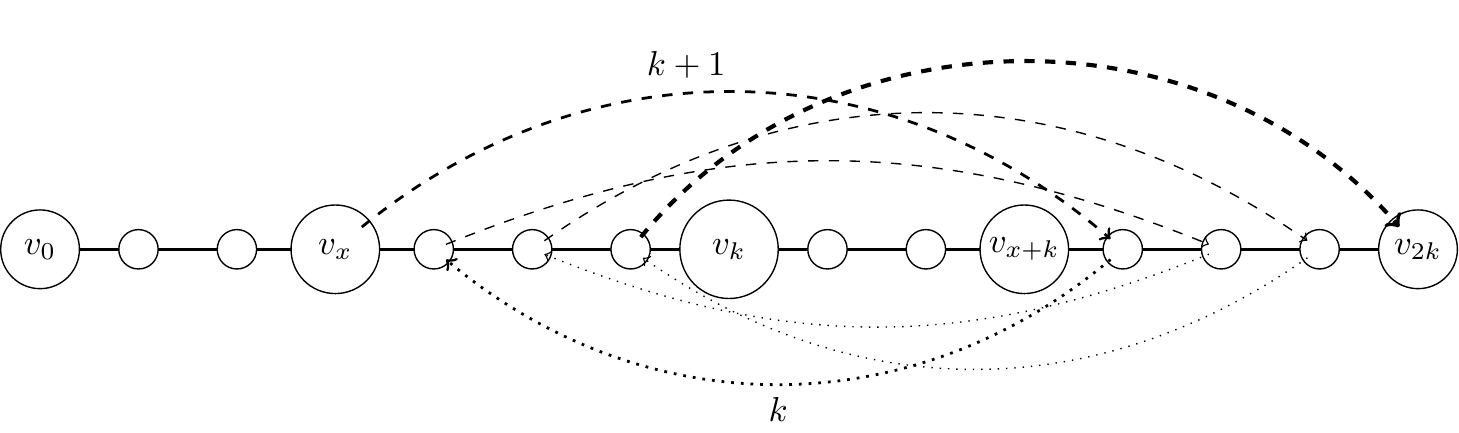}
    \includegraphics[width=.9\textwidth]{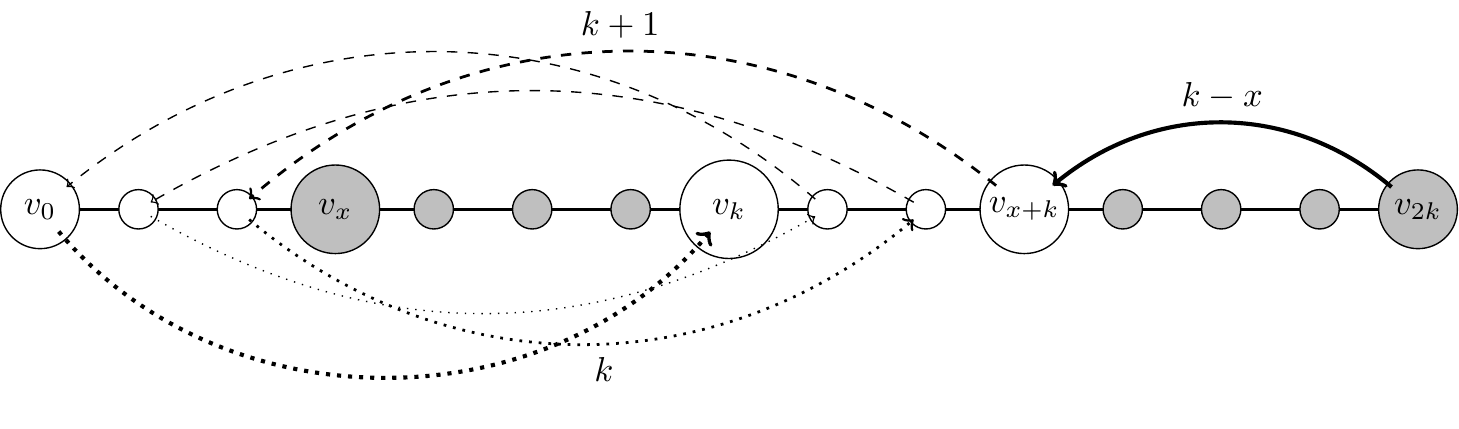}
    \caption{Nonadaptive strategy for odd paths, shown in two phases.  Essentially alternating steps of size $k+1$ (dashed, drawn above the path) and $k$ (dotted, drawn below)}
    \label{figure:odd_paths}
\end{figure}

In the first phase, the token begins at $v_x$, and the Explorer alternates between declaring $k+1$ and $k$ until the token arrives at $v_{2k}$.  At this point, the token will have seen every vertex $v_t$ where $x \leq t < k$ as well as vertices where $x+k < t \leq 2k$, and the token is at $v_{2k}$ (as shaded gray in Figure \ref{figure:odd_paths}).

The second phase of the strategy begins with the token at $v_{2k}$, and the Explorer declares $k-x$ to move the token to $v_{x+k}$.  After this, the Explorer simply alternates between declaring $k+1$ and $k$ (mirroring the path in phase one) until they arrive at $v_0$.  Finally, declaring $k$ moves the token to $v_k$, which is the last vertex that needed to be visited.
\end{proof}

\section{Conclusion and Further Research}

We conclude this paper with a discussion on further research. We next define a variant using distance norms for the director to move, such as allowing the director to move along any path, denoted $f_p(G,v)$ or even trail, denoted $f_t(G,v)$ a given length. Notice that a walk would not be an interesting variant. Interestingly, the number of visited vertices when allowing the token to move along paths, $f_p(G,v)$ can be higher than $f_d(G,v)$ as the Explorer can now call significantly higher values.

Another variant we propose is that of a weighted graph. It is shown in the arXiv Companion that every graph has a weighting such that the set of vertices visited at the end of the game is the entire graph. Further work in this variant could include determining the complexity, adding constraints on the vertex weights, or classifying the sets of weights that allow all vertices to be visited in terms of other graph parameters.


One can also consider the number of steps required to visit $f_d(G,v)$ vertices for various graph families, as was done in Theorem \ref{nonadaptive-paths} as well as \cite{HPW, N10, CLST, N14}. This then furthers the research into determining the complexity of the problem. Similarly, further work into what graph families allow for non-adaptive strategies is of interest. Finally, there are clearly many more interesting graph families to consider such as hypercubes or even random graphs.  

\subsection*{Acknowledgements}

The authors would like to thank the organizers and participants of the Polymath Research Experience for Undergraduates where this work was conducted.

\bibliographystyle{abbrv}
\bibliography{EDbib}

\begin{thebibliography}{1}

\bibitem{CLST}
L.-J. Chen, J.-J. Lin, M.-Z. Shieh, and S.-C. Tsai.
\newblock More on the magnus-derek game.
\newblock {\em Theoretical Computer Science}, 412(4):339 -- 344, 2011.

\bibitem{G}
D.~Gerbner.
\newblock The magnus-derek game in groups.
\newblock {\em Discret. Math. Theor. Comput. Sci.}, 15:119--126, 2013.

\bibitem{HPW}
C.~A. Hurkens, R.~A. Pendavingh, and G.~J. Woeginger.
\newblock The magnus-derek game revisited.
\newblock {\em Information Processing Letters}, 109(1):38 -- 40, 2008.

\bibitem{janakiraman2012structural}
T.~Janakiraman, P.~Alphonse, and V.~Sangeetha.
\newblock Structural properties of k-distance closed domination critical graphs
  for k= 5 and 6.
\newblock {\em International Journal of Engineering Science Advanced Computing
  and Bio-Technology}, 3:1--14, 2012.

\bibitem{Knuth}
D.~E. Knuth.
\newblock {\em The Art of Computer Programming, Volume 1 (3rd Ed.): Fundamental
  Algorithms}.
\newblock Addison Wesley Longman Publishing Co., Inc., USA, 1997.

\bibitem{N10}
Z.~Nedev.
\newblock An o(n)-round strategy for the magnus-derek game.
\newblock {\em Algorithms}, 3:244--254, 09 2010.

\bibitem{N14}
Z.~Nedev.
\newblock A reduced computational complexity strategy for the magnus-derek
  game.
\newblock {\em International Mathematical Forum}, 9:325--333, 2014.

\bibitem{NM}
Z.~Nedev and S.~M. Muthukrishnan.
\newblock The magnus-derek game.
\newblock {\em Theoretical Computer Science}, 393(1):124 -- 132, 2008.

\bibitem{AD}
D.~T. Ásgeirsson and P.~Devlin.
\newblock Palindromes in finite groups and the explorer-director game, 2019.

\end{thebibliography}

\end{document}